\documentclass{amsart}
\usepackage{amsbsy,amssymb,amscd,amsmath,amsthm, hyperref}
\usepackage{a4wide}
\usepackage[usenames,dvipsnames,svgnames,table]{xcolor}

\newcommand{\out}[1]{\ }

\newcommand{\RR}{{\mathbb R}}
\newcommand{\CC}{{\mathbb C}}

\renewcommand{\ge}{\geqslant}
\renewcommand{\le}{\leqslant}
\renewcommand{\phi}{\varphi}
\renewcommand{\epsilon}{\varepsilon}

\renewcommand{\atop}[2]{\genfrac{}{}{0pt}{}{#1}{#2}}

\DeclareMathOperator{\PSH}{PSH}

\makeatletter
\newtheorem*{rep@theorem}{\rep@title}
\newcommand{\newreptheorem}[2]{%
\newenvironment{rep#1}[1]{%
 \def\rep@title{#2 \ref{##1}}%
 \begin{rep@theorem}}%
 {\end{rep@theorem}}}
\makeatother

\newtheorem{theorem}{Theorem}[section]
\newreptheorem{theorem}{Theorem}

\theoremstyle{definition}
\newtheorem{definition}[theorem]{Definition}

\theoremstyle{remark}

\numberwithin{equation}{section}

\newtheoremstyle{case}
{3pt}
  {3pt}
  {}
  {}
  {\bfseries}
  {:}
  {.5em}
  {}
\theoremstyle{case}

\numberwithin{subcase}{case}

\usepackage[utf8]{inputenc}
\begin{document}

\title{An example concerning Sadullaev's boundary relative extremal functions}

\dedicatory{In memory of J\'ozef Siciak}
\author{Jan Wiegerinck}
\address{KdV Institute for Mathematics
\\University of Amsterdam
\\Science Park 105-107
\\P.O. box 94248, 1090 GE Amsterdam
\\The Netherlands}
\email{j.j.o.o.wiegerinck@uva.nl}
\begin{abstract}
We exhibit a smoothly bounded domain $\Omega$ with the property that for suitable $K\subset\partial \Omega$ and $z\in \Omega$ the \emph{Sadullaev boundary relative extremal functions} satisfy the inequality $\omega_1(z,K,\Omega)<\omega_2(z,K,\Omega)\le \omega(z,K,\Omega)$. 
\end{abstract}
\maketitle
\section{Introduction}
In \cite{Sa} Sadullaev introduced several so-called \emph{boundary relative extremal functions} for compact sets $K$ in the boundary of domains $D\subset\CC^n$, and asked whether their regularizations are perhaps always equal. Recently Djire and the author \cite{IbJa, IbJa1} gave a positive answer in certain cases where $D$ and $K$ are particularly nice. 

In this note we show that in general equality does not hold. The example is formed by a suitable compact set in the boundary of the domain $\Omega$ that was constructed by Forn\ae ss and the author \cite{FoWi} as an example of a domain $D$ where bounded plurisubharmonic functions that are continuous on $D$ cannot be approximated by plurisubharmonic functions that are continuous on $\overline D$. We start by briefly recalling the definitions of boundary relative extremal functions and the construction of the domain $\Omega$.

\subsection{Boundary relative extremal functions} We follow Sadullaev \cite[Section 27]{Sa}. 
Let $D$ be a domain with smooth boundary in $\CC^n$, $\xi\in\partial D$, and $A_\alpha(\xi)=\{z\in D; |z-\xi|<\alpha\delta_\xi(z)\}$, where $\alpha\ge 1$ and $\delta_\xi(z)$ is the distance from $z$ to the tangent plane at $\xi$ to $\partial D$.
For a function $u$ defined on $D$, put
\[\tilde u(\xi)=\sup_{\alpha>1} \limsup_{\atop{z\to\xi}{z\in A_\alpha(\xi)}} u(z),\quad \xi\in\partial D.\]
\begin{definition}Let $\PSH(D)$ denote the plurisubharmonic functions on $D$ and let $K\subset\partial D$ be compact. We define the following  \emph{boundary relative extremal functions}
\begin{enumerate}
\item \[\omega(z,K,D)=\sup\{u(z): u\in\PSH(D), u\le 0, \tilde u|_K\le -1\};\]
\item \[\omega_1(z,K,D)=\sup\{u(z): u\in\PSH(D)\cap C(\overline D), u\le 0, u|_K\le -1\};\]
\item \[\omega_2(z,K,D)=\sup\{u(z): u\in\PSH(D), u\le 0, \limsup_{\atop{z\to\xi}{z\in D}}\le -1,\text{ for all } \xi\in K\}.\]
\end{enumerate}
\end{definition}
The upper semi-continuous regularization $u^*$ of a function $u$ on a domain $D$ is defined as
\[u^*(z)=\limsup_{w\to z} \{u(w)\}.\]
The functions $\omega^*$, $\omega_1^*$, $\omega_2^*$ are plurisubharmonic.
Observing that $\omega_1(z,K,D)\le\omega_2(z,K,D)\le\omega(z,K,D)$, Sadullaev's question is \emph{for what $j$ is $\omega^*(z,K,D)\equiv\omega^*_j(z,K,D)$?}

\subsection{The domain $\Omega$}
We briefly recall the construction and properties of the domain $\Omega$ from \cite{FoWi}.
\begin{equation}
\Omega=\{(z,w)\in\CC^2; |w-e^{i\phi(|z|)}|^2<r(|z|)\}.
\end{equation}
Here $r$ and $\phi$ are in $\CC^\infty(\RR)$ with the following properties: $-1\le r\le 2$; $r(t)\le 0$ for $t\le 1$ and for $t\ge 17$; $r(t)\equiv 1$ for $3\le t\le 8$ and for $10\le t\le 15$; $r(t)$ takes its maximum value $=2$ precisely at $t=2$, 9, and 16. Moreover, $r'(t)>0$ on $1\le t<2$, $8<t<9$ and $15<t<16$, while $f'(t)<0$ on  $2<t<3$, $9<t<10$, and $16<t\le 17$. Next $\phi$ satisfies 
$\phi(t)<-\pi/2$ for $t\le 4$ and for $t\ge 14$; $\phi(t)> \pi/2+100$ for $5\le t\le6 $ and for $12\le t\le 13$ and $\phi(t)<-\pi/2+100$ for $7<t<10$, and we demand in addition that $\phi\le 108$. 

From \cite{FoWi} we recall that $\Omega$ is a Hartogs domain with smooth boundary, and that the annulus 
\begin{equation}
A=\{(z,w);w=0, 2\le |z|\le 15 \}
\end{equation}
is contained in $\overline\Omega$.
\section{Negative answer to Sadullaev's question}
\begin{theorem}Let $K=\{(z,w\in\partial \Omega; |z|=2 \text{ or } |z|=16\}$. Then 
\[\omega_1((z,w),K,\Omega)<\omega_2((z,w),K,\Omega)\]
for $(z,w)$ in an open neighborhood of $\{w=0, |z|=9\}$.
\end{theorem}
\begin{proof}
Let $u\in\PSH(\Omega)\cap C(\overline\Omega)$, $u\le 0$, $u|_K\le -1$. Then by the maximum principle, $|u|\le -1$ on the discs $|w-e^{i\phi(|z|)}|\le 2$, where $z$ is fixed and satisfies $|z|=2$ or  $|z|=16$, and in particular on the circles $C_1(w)=\{(z, w): |z|=2\}$ and $C_2(w)=\{(z,w): |z|=16\}$, where $|w|<1$.  Because $\Omega$ is a smoothly bounded domain, it follows from  \cite[Theorem 1]{FoWi} (see also \cite{PeWi} for recent extensions of this theorem), that $u$ can be approximated uniformly on $\overline\Omega$ by smooth plurisubharmonic functions $v$ defined on shrinking neighborhoods of $\overline\Omega$. 

Let $\Omega_\delta=\{\zeta\in\CC^2; d(\zeta,\overline \Omega)<\delta\}$.  Then given $\epsilon>0$, there exist $\delta>0$ and $v\in \PSH(\Omega_\delta)$, such that $|u-v|<\epsilon$ on $\overline\Omega$. For $|w|<\delta$ the annulus $A_w=\{(z,w): 2\le |z|\le 16\}$ is contained in $\Omega_\delta$. On its boundary, which equals $C_1(w)\cup C_2(w)$,  we have that $v<-1+\epsilon$, hence this also holds on $A_w$. It follows that $u<-1 +2\epsilon$ on $A_w\cap\overline\Omega$, in particular $u<-1+2\epsilon$ on the open set
$V=\{(z,w): 8<|z|<10, |w|<\delta, |w|<r(|z|)-1 \}\subset\Omega$. 
It follows that $\omega_1((z,w),K,\Omega)\le-1+2\epsilon$ on $V$, and therefore also $\omega_1^*((z,w),K,\Omega)\le-1+2\epsilon$ on $V$.

Next we will construct a plurisubharmonic function in the family that determines $\omega_2$. The construction is as in \cite[Section 2]{FoWi}.  On $\Omega \cap(\{3<|z|<8\}\cup\{10<|z|<15\}$ there exists a continuous branch of $\arg w$, denoted by $h(z,w)$, such that
\[\phi(z)-\pi/2\le h(z,w)\le \phi(z)+\pi/2.\]
In \cite{FoWi} we constructed the following plurisubharmonic function.
\begin{equation}
f(z,w)=\begin{cases} 0&\text{if } |z|<4 \text{ or if } |z|>14\\
\max\{0, h(z,w)\}&\text{if } 3<|z|<6 \text{ or if } 12<|z|<14\\
\max\{100,h(z,w)\}&\text{if } 5<|z|<8 \text{ or if } 10<|z|<13\\
100 &\text{if } 7<|z|<11.
\end{cases}
\end{equation}
It satisfies $f\le 110$ on $\Omega$, $f\equiv0$ on $\{|z|\le 3\}$ and on $\{|z|\ge 14\}$, hence $f$ extends continuously by 0 to $\overline\Omega\cap (\{|z|\le 3\}\cup \{|z|\ge 14\})$, and $f=100$ on $V$. 
The plurisubharmonic function $g$ on $\Omega$ defined by\[g(\zeta)= \frac{f(\zeta)-110}{110},\quad (\zeta=(z,w))\]
is negative, identically equal to $-1$ on $\overline\Omega\cap (\{|z|\le 3\}\cup \{|z|\ge 14\})$, and equal to  $-10/11$ on $V$. Hence also $\omega_2^*((z,w),K,\Omega)\ge\omega_2((z,w),K,\Omega)\ge -10/11$ on $V$.
Choosing $\epsilon< 1/10$ completes the proof.
\end{proof}

\end{document}